\documentclass[12pt]{scrartcl}
\usepackage[english]{babel}
\usepackage[utf8x]{inputenc}
\usepackage{amsfonts}
\usepackage{amsmath}
\usepackage{amssymb}
\usepackage{amsthm}
\usepackage{mathrsfs}
\usepackage[colorlinks=false,urlcolor=blue]{hyperref}
\usepackage{breakurl}
\usepackage{listings}
\usepackage{graphicx}
\usepackage{enumerate}
\usepackage{tikz}
\usetikzlibrary{patterns}
\usepackage{pdflscape}
\usepackage{afterpage}

\addtokomafont{section}{\rmfamily}
\addtokomafont{subsection}{\rmfamily}
\addtokomafont{paragraph}{\rmfamily}
\usepackage{indentfirst}

\swapnumbers
\theoremstyle{plain}
\newtheorem{satz}{Theorem}[section]
\newtheorem{lem}[satz]{Lemma}
\newtheorem{kor}[satz]{Corollary}

\theoremstyle{definition}
\newtheorem{defn}[satz]{Definition}
\newtheorem{bem}[satz]{Remark}

\newcommand{\R}{\mathbb{R}}%real numbers
\newcommand{\C}{\mathbb{C}}%complex numbers
\newcommand{\N}{\mathbb{N}}%natural numbers
\newcommand{\Z}{\mathbb{Z}}%integers
\renewcommand{\O}{\mathbb{O}}%octonions
\newcommand{\Id}{\operatorname{Id}}%identity operator
\newcommand{\Ric}{\operatorname{Ric}}%Ricci tensor
\newcommand{\scal}{\operatorname{scal}}%scalar curvature
\newcommand{\vol}{\operatorname{vol}}%volume form
\newcommand{\Sym}{\operatorname{Sym}}%symmetrized tensor power
\newcommand{\tr}{\operatorname{tr}}%trace
\newcommand{\diag}{\operatorname{diag}}%diagonal matrix
\newcommand{\End}{\operatorname{End}}%endomorphisms
\newcommand{\Aut}{\operatorname{Aut}}%automorphisms
\newcommand{\Ad}{\operatorname{Ad}}%adjoint representation of Lie group
\newcommand{\ad}{\operatorname{ad}}%adjoint representation of Lie algebra
\newcommand{\Cas}{\operatorname{Cas}}%Casimir operator
\newcommand{\Sy}{\mathscr{S}}%symmetric tensor fields
\newcommand{\SU}{\operatorname{SU}}%special unitary group
\newcommand{\su}{\mathfrak{su}}%Lie algebra of special unitary group
\newcommand{\SO}{\operatorname{SO}}%special orthogonal group
\newcommand{\so}{\mathfrak{so}}%Lie algebra of special orthogonal group
\newcommand{\Spin}{\operatorname{Spin}}%spin group
\renewcommand{\Re}{\operatorname{Re}}%real part
\renewcommand{\Im}{\operatorname{Im}}%imaginary part
\newcommand{\pr}{\operatorname{pr}}%projection
\renewcommand{\H}{\mathcal{H}}%calligraphic H (horizontal distribution)
\newcommand{\m}{\mathfrak{m}}%reductive complement
\newcommand{\h}{\mathfrak{h}}%Lie algebra of H
\newcommand{\g}{\mathfrak{g}}%Lie algebra of G
\renewcommand{\k}{\mathfrak{k}}%Lie algebra of K
\newcommand{\TT}{\Sy^2_{\mathrm{tt}}}%tt-tensors
\newcommand{\e}{\mathfrak{e}}%exceptional Lie algebra
\newcommand{\f}{\mathfrak{f}}%exceptional Lie algebra
\renewcommand{\t}{\mathfrak{t}}%torus Lie algebra
\newcommand{\V}{\mathfrak{v}}%gothic v
\newcommand{\Weyl}{\mathfrak{W}}%Weyl tensors
\newcommand{\cartan}{\boxdot}%Cartan product
\newcommand{\A}{\mathcal{A}}%calligraphic A
\newcommand{\LC}{\nabla}%Levi-Civita connection
\newcommand{\CR}{\bar\nabla}%canonical reductive connection
\newcommand{\Riem}{R}%Riemannian curvature
\newcommand{\Rcr}{\bar R}%curvature of CR connection
\newcommand{\Tcr}{\bar T}%torsion of CR connection
\newcommand{\PSO}{\operatorname{PSO}}%pr. special orthogonal group
\newcommand{\Ricr}{\overline{\operatorname{Ric}}}%Ricci of CR connection
\newcommand{\LiE}{LiE}%LiE software package
\newcommand{\SI}{\mathcal{S}}%Einstein-Hilbert functional
\newcommand{\E}{\mathrm{E}}%exceptional Lie group E
\newcommand{\Der}{\operatorname{Der}}%extension as derivation

\title{\rmfamily Stability of the Non--Symmetric Space $\E_7/\PSO(8)$}
\author{Paul Schwahn\footnote{Corresponding author. Institut f\"ur Geometrie und Topologie, Fachbereich Mathematik, Universit\"at Stuttgart, Pfaffenwaldring 57, 70569 Stuttgart, Germany. E-mail: paul.schwahn@mathematik.uni-stuttgart.de}, Uwe Semmelmann\footnote{Institut f\"ur Geometrie und Topologie, Fachbereich Mathematik, Universit\"at Stuttgart, Pfaffenwaldring 57, 70569 Stuttgart, Germany. E-mail: uwe.semmelmann@mathematik.uni-stuttgart.de}, Gregor Weingart\footnote{Instituto de Matem\'aticas, Universidad Nacional Aut\'onoma de M\'exico, Avenida Universidad s/n, Colonia Lomas de Chamilpa, 62210 Cuernavaca, Mexico. E-mail: gw@im.unam.mx}}
\date{\today}

\begin{document}

\maketitle

\begin{abstract}
\noindent
We prove that the normal metric on the homogeneous space $\E_7/\PSO(8)$ is stable with respect to the Einstein-Hilbert action, thereby 
exhibiting the first known example of a non-symmetric metric of positive scalar curvature with this property.

\medskip

\noindent{\textit{Mathematics Subject Classification} (2020): 53C24, 53C25, 53C30}

\medskip

\noindent{\textit{Keywords:} Homogeneous spaces, Einstein metrics, Stability}
\end{abstract}

\subsection*{Declarations}

\paragraph*{Funding.} The first and second author acknowledge the support received by the Special Priority Program SPP 2026 \emph{Geometry at Infinity} funded by the Deutsche Forschungsgemeinschaft DFG. Likewise the third author expresses his gratitude for the funding received as a SNI member from the Consejo Nacional de Ciencia y Tecnolog\'{i}­a CONACyT.
%Grant numbers
%This work was supported by the National Institutes of Health [grant numbers xxxx, yyyy]; the Bill & Melinda Gates Foundation, Seattle, WA [grant number zzzz]; and the United States Institutes of Peace [grant number aaaa].

\paragraph*{Declarations of interest.} None.

\pagebreak

\section{Introduction}
\label{sec:intro}

Einstein metrics are Riemannian or pseudo-Riemannian metrics whose Ricci tensor is proportional to the metric, i.e.~$\Ric_g=Eg$ for some constant $E$ called the Einstein constant of $g$.  It is a well-known fact that Einstein metrics on closed manifolds are precisely the critical points of the Einstein--Hilbert functional $\SI(g):=\int_M\scal_g\vol_g$ restricted to the space of metrics of unit volume. Einstein metrics are always saddle points but they can be local maxima if the functional is further restricted to the set of unit volume metrics with constant scalar curvature. Tangent to this is the space of tt-tensors, i.e.~traceless and divergence-free symmetric $2$-tensors. The second variation of the Einstein-Hilbert functional $\SI$ on tt-tensors can be expressed in terms of the Lichnerowicz Laplacian $\Delta_L$ on symmetric $2$-tensors as
\[\SI_g''(h,h)=-\big(\Delta_Lh-2Eh,h\big)_{L^2}.\]
Following Koiso \cite{Koiso80} we will call an Einstein metric $g$ \emph{stable} if $g$ is a local maximum of the Einstein-Hilbert functional $\SI$ restricted to the space of tt-tensors. In particular this is the case if $\SI''_g<0$ on tt-tensors, or equivalently if $\Delta_L>2E$. If $g$ is a saddle point instead, the metric $g$ is called \emph{unstable}. The existence of a tt-tensor $h$ such that $\SI''_g(h,h)>0$, or equivalently, $\Delta_Lh=\mu h$ for some eigenvalue $\mu<2E$, implies instability of the metric. These eigentensors for eigenvalues less than $2E$ are also called \emph{destabilizing directions}. Unstable Einstein metrics are particularly
 interesting since they turn out to also be unstable with respect to Perelman's $\nu$-entropy as well as dynamically unstable with respect to the Ricci flow (see \cite{CH15},\cite{K15}). Finally, metrics $g$ with $\SI''_g\leq0$ on tt-tensors, or equivalently $\Delta_L\ge2E$, will be called \emph{linearly stable}.

In \cite{Koiso80} Koiso studied the stability question for symmetric spaces. It turned out that most of the irreducible symmetric spaces of compact type are linearly stable and only very few are unstable (see also \cite{SW22} and \cite{S22} for the proof in the cases not covered by Koiso). Further examples of stable Einstein metrics are provided by Einstein metrics of negative sectional curvature (see \cite[Cor.~12.73]{B87}), or by Kähler--Einstein metrics of negative scalar curvature (see \cite{DWW07}). All known compact manifolds of vanishing Ricci curvature, in other words all manifolds admitting parallel spinors, are linearly stable (see \cite{DWW05}). On the other side there are many examples of unstable Einstein metrics, e.~g.~metrics on the total space of a Riemannian submersion over an unstable base (see \cite{B05},\cite{WW21}). Sometimes destabilizing directions are related to harmonic forms, as on Kähler--Einstein manifolds of positive scalar curvature with $b_2>0$ (see \cite{CHI04}), nearly Kähler manifolds in dimension $6$ with $b_2>0$ or $b_3>0$ (see \cite{SWW20}), or on Einstein--Sasaki manifolds with $b_2>0$ (see \cite{SWW22}). Recently, many more unstable examples on homogeneous spaces appeared in the work of J.~Lauret et al. (see \cite{L1},\cite{L2},\cite{L3}). It is interesting to note that all these unstable examples have positive scalar curvature. Indeed it is rather surprising that so far, except on the symmetric spaces, no example of a stable Einstein metric of positive scalar curvature was found.

In this article we will consider the generalized Wallach space $\E_7/\PSO(8)$ and its universal cover. The standard metric on this non-symmetric homogeneous space induced by minus the Killing form is known to be Einstein of positive scalar curvature. Moreover it was shown in \cite{L2} that the standard metric is $G$-stable in the sense that it is a local maximum of the Einstein--Hilbert functional $\SI$ restricted to the space of $G$-invariant metrics. The main result of our article is the stability of the standard metric on $\E_7/\PSO(8)$ in the much larger class of all Riemannian metrics. This provides the first example of a stable non-symmetric Einstein metric of positive scalar curvature.
 
\begin{satz}[Lower Estimate for the Lichnerowicz Laplacian]
\hfill\label{lest}\break
Let $g$ be the standard Riemannian metric of positive Einstein constant $E=\frac{\scal}{105}=\frac{13}{36}$ on the connected homogeneous space $M=\E_7/\PSO(8)$ or its universal cover. Then the Lichnerowicz Laplacian $\Delta_L$ restricted to the space of tt-tensors satisfies
\[\Delta_L\geq\frac{30}{13}E>2E.\]
Equality is realized exactly on left invariant, trace free symmetric $2$-tensors. In particular, the Riemannian metric $g$ is a stable, non-symmetric Einstein metric of positive scalar curvature.
\end{satz}

The proof of the main theorem rests on an estimate of $\Delta_L$ against a curvature term $q(R)$. The strategy of the article is as follows. Section~\ref{sec:prelim} sets up the necessary preliminaries about the Lichnerowicz Laplacian, normal homogeneous spaces and Casimir operators and introduces along the way the auxiliary operator $\A^\ast\A$ that depends on the torsion of the reductive connection on a homogeneous space. In Section~\ref{sec:curv}, the curvature endomorphism $q(R)$ is related to $\A^\ast\A$ and a formula for the latter is given in terms of Casimir operators. The Lie algebra $\e_7$ as well as the homogeneous space $\E_7/\PSO(8)$ and its relevant structure are discussed in Section~\ref{sec:e7pso8}. Finally, in Section~\ref{sec:computation} we compute the eigenvalues of $\A^\ast\A$ and thus $q(R)$, yielding a sufficient lower bound on $\Delta_L$ to prove Theorem~\ref{lest}.

\section{Preliminaries}
\label{sec:prelim}

\subsection{The Lichnerowicz Laplacian}
\label{sec:prelimll}

Let $(M,g)$ be a Riemannian manifold with Levi-Civita connection denoted by $\LC$. The Riemannian curvature tensor and Ricci tensor are given by
\begin{align*}
\Riem(X,Y)Z&:=\LC_X\LC_YZ-\LC_Y\LC_XZ-\LC_{[X,Y]}Z,\\
\Ric(X,Y)&:=\tr(Z\mapsto\Riem(Z,X)Y).
\end{align*}
We use the term \emph{tensor bundle} to refer to a vector bundle $VM$ that is associated to the frame bundle of $(M,g)$ by some representation of $\SO(n)$. Equivalently, a tensor bundle is a $\SO(TM)$-invariant subbundle of some tensor power of $TM$.

On any tensor bundle $VM$ the \emph{standard curvature endomorphism} is the symmetric endomorphism $q(\Riem)\in\End(VM)$ defined by
\[q(\Riem):=\sum_{i<j}(e_i\wedge e_j)_\ast R(e_i,e_j)_\ast,\]
where $(e_i)$ is a local orthonormal frame of $TM$. The asterisk denotes the natural action of $\so(T)$ on tensors, i.e. extension as a derivation. We also implicitly identify $\Lambda^2T\cong\so(T)$ via
\[X\wedge Y\longmapsto(Z\mapsto g(X,Z)Y-g(Y,Z)X).\]
On $TM$ the endomorphism  $q(R)$ coincides with the Ricci endomorphism, i.e.
\[g(q(R)X,Y)=\Ric(X,Y).\]
Applied to the bundle $\Sym^2T^\ast M$ of symmetric $2$-tensors, $q(R)$ can be written as
\[q(R)=-2\mathring{\Riem}-\Der_{\Ric},\]
where $\mathring{\Riem}$ is the so-called \emph{curvature operator of second kind} given by
\[(\mathring{\Riem}h)(X,Y)=\sum_ih(R(e_i,X)Y,e_i),\qquad h\in\Sym^2T^\ast M,\]
while $\Der_A$ denotes the extension of some endomorphism $A\in\End(T)$ to higher-rank tensors as a derivation.

The \emph{Lichnerowicz Laplacian} $\Delta_L$ is now defined on tensor fields, i.e.~smooth sections of $VM$, by
\[\Delta_L:=\LC^\ast\LC+q(\Riem).\]
This is a Laplace type operator with a discrete spectrum accumulating only at positive infinity. On differential forms $\Delta_L$ coincides with the Hodge Laplacian $\Delta=d^\ast d+dd^\ast$, thus generalizing the latter to tensors of arbitrary algebraic type. Even more generally, the Lichnerowicz Laplacian is an instance of the \emph{standard Laplace operator} on geometric vector bundles introduced in \cite{SW19}. 

Since we aim to investigate the spectrum of $\Delta_L$ on tt-tensors, the bundle we will primarily consider is $\Sym^2T^\ast M$. We will denote by
\[\Sy^p(M):=\Gamma(\Sym^pT^\ast M),\quad p\in\N,\]
the space of smooth sections of $\Sym^pT^\ast M$.

The \emph{divergence operator} on symmetric tensors is defined as the metric contraction of the covariant derivative, i.e.
\[\delta:\ \Sy^{p+1}(M)\to\Sy^p(M):\ h\mapsto\delta h:=-\sum_ie_i\lrcorner\nabla_{e_i}h\]
for a local orthonormal frame $(e_i)$ of $TM$.

As explained in the introduction, the stability of an Einstein metric $g$ is decided by a spectral property of the Lichnerowicz operator $\Delta_L$ on the space $\TT(M)$ of tt-tensors, i.e.~on symmetric $2$-tensors $h$ with $\tr_gh=0$ and $\delta h=0$. On this space we have the lower bound 
\begin{equation}
\Delta_L\geq2q(R)
\label{eq:weitzenboeck}
\end{equation}
(see \cite[Prop.~6.2]{HMS16}), which will be the main tool for our proof of the stability of the standard metric on $\E_7/\PSO(8)$. The estimate is consequence of the Weitzenböck formula
\[\Delta_L-2q(R)=\LC^\ast\LC-q(R)=\delta\delta^\ast-\delta^\ast\delta,\]
where the symmetrized covariant derivative (or \emph{Killing operator}) $\delta^\ast:\ \Sy^2(M)\to\Sy^3(M)$ is formally adjoint to the divergence $\delta$. Tensors in the kernel of $\delta^\ast$ are called \emph{Killing tensors} (see \cite{HMS16} for further details). We see that a divergence-free tensor $h$ satisfies the equality $\Delta_Lh=2q(R)h$ if and only if it is Killing. In many cases, e.g.~for the Berger space $\SO(5)/\SO(3)_{\mathrm{irr}}$ (see \cite{SWW22}), destabilizing directions for Einstein metrics are realized by Killing tensors.

\subsection{Normal homogeneous spaces}
\label{sec:prelimnh}

Let $M=G/H$ be a homogeneous space and let $\g$ and $\h$ denote the Lie algebras of $G$ and $H$, respectively. Let further $Q$ be an $\Ad(G)$-invariant inner product on $\g$, and let $\m:=\h^{\perp_Q}$ denote the $Q$-orthogonal complement of $\h$ in $\g$, which is canonically identified with the tangent space $T_oM$ at the base point $o=eH$. In particular, we obtain an $\Ad(H)$-invariant decomposition $\g=\h\oplus\m$. We will use subscripts $X_\h,X_\m$ to denote the projection of $X\in\g$ to the respective direct summand. The unique $G$-invariant Riemannian metric $g$ which coincides with the restriction $Q\big|_{\m}$ at the base point is called the \emph{normal} metric induced by $Q$.

For compact and semisimple $G$, the Killing form $B_\g$ is negative-definite -- hence, $-B_\g$ is an $\Ad(G)$-invariant inner product on $\g$. The metric $g$ on $M$ induced by $-B_\g$ will be called the \emph{standard metric}. Naturally, if $G$ is simple, every normal metric will be a scalar multiple of the standard metric.

A normal homogeneous space is in particular naturally reductive, that is, it satisfies
\[g([X,Y]_\m,Z)+g(Y,[X,Z]_\m)=0\quad\text{for all }X,Y,Z\in\m.\]
In other words, the $G$-invariant $(2,1)$-tensor $\A$ defined by $\A_XY:=[X,Y]_\m$ is totally skew-symmetric.

Since $\m\subset\g$ is $\Ad(H)$-invariant, the decomposition is \emph{reductive} -- therefore, it defines a $G$-invariant connection $\CR$ on $M$, called the \emph{canonical, reductive} (or \emph{Ambrose--Singer}) connection, which is induced by the left-invariant principal connection
\[\pr_\h\circ\,\theta:\ TG\longrightarrow\h,\]
where $\theta:\ TG\to\g$ denotes the Maurer-Cartan form and $\pr_\h$ some $H$-equivariant projection from $\g$ to $\h$. It can also be viewed as the affine Ehresmann connection corresponding to the horizontal distribution $\H=\bigcup_{x\in G}dl_x(\m)$ in $TG$. A distinctive property of the reductive connection is that every $G$-invariant tensor is $\CR$-parallel. The $G$-invariant torsion and curvature tensors of $\CR$ are given by
\begin{equation}
\begin{aligned}
\Tcr(X,Y)&=-[X,Y]_\m=-\A_XY,\\
\Rcr(X,Y)Z&=-[[X,Y]_\h,Z]
\end{aligned}
\qquad\text{for }X,Y,Z\in\m.\label{eq:crcurv}
\end{equation}
In particular $\CR$ is a metric connection with parallel and totally skew-symmetric torsion. If we extend the endomorphism $\A_X\in\so(\m)$ to tensors as a derivation $(\A_X)_\ast$, it induces a $\CR$-parallel bundle map
\begin{alignat*}{2}
\A&:\ VM\to T^\ast M\otimes VM:\quad &v&\mapsto\sum_ie_i^\flat\otimes(\A_{e_i})_\ast v\\
\intertext{for any tensor bundle $VM$, with metric adjoint}\\
\A^\ast&:\ T^\ast M\otimes VM\to VM:\quad &\alpha\otimes v&\mapsto\sum_i\alpha(e_i)(\A_{e_i})_\ast v.
\end{alignat*}
This allows us to express the relation between the reductive connection and the Levi-Civita connection of $g$ by
\[\LC=\CR+\frac{1}{2}\A.\]
Recall that if $(M,g)$ is a Riemannian symmetric space, it satisfies the Cartan relation $[\m,\m]\subset\h$, implying $\Tcr=0$ and thus $\LC=\CR$. In this sense, the tensor $\A$ measures the failure of a normal homogeneous space $(M,g)$ to be (locally) symmetric.

It is worth noting that the standard Laplacian $\CR^\ast\CR+q(\Rcr)$ coincides with the action of the Casimir operator (see Section~\ref{sec:prelimcas}) on the left-regular representation of $G$ on sections of tensor bundles over $M$. This fact has been vital for Koiso's study of the stability of compact symmetric spaces \cite{Koiso80}.

We further note that the composition $\A^\ast\A$ is a $\CR$-parallel self-adjoint bundle endomorphism of $VM$ that can, by combining the above, be written as
\[\A^\ast\A=-\sum_i(\A_{e_i})_\ast^2.\]
This auxiliary operator will be employed in order to compute the spectrum of $q(R)$ on the symmetric $2$-tensors of the normal homogeneous space $\E_7/\PSO(8)$, utilizing the formulae in Section~\ref{sec:curv}.

\subsection{Casimir operators}
\label{sec:prelimcas}

The leitmotif of analysis and geometry on normal homogeneous spaces is to reduce calculations as far as possible to the computation of eigenvalues of Casimir operators. Fix some invariant inner product $Q$ on a compact Lie algebra $\g$. Given a representation $(V,\rho_\ast)$ of $\g$, its \emph{Casimir operator} is the endomorphism defined by
\[\Cas^{\g,Q}_V:=-\sum_i\rho_\ast(e_i)^2\in\End(V).\]
This operator is  $\g$-equivariant. By Schur's Lemma it hence acts as multiplication with a constant when applied to an finite-dimensional irreducible complex representation of $\g$. This constant can be computed by means of Freudenthal's formula. Choose a maximal torus $\t\subset\g$ and let $\langle\cdot,\cdot\rangle$ be the inner product on the dual $\t^\ast$ that is induced by $Q\big|_{\t\times\t}$. We label the (equivalence classes of) finite-dimensional irreducible representations $V_\gamma$ of $\g$ by their highest weights $\gamma\in\t^\ast$. If the complex representation $V_\gamma$ has a real structure, we will sometimes abuse notation and denote the real form by $V_\gamma$ as well. The Casimir eigenvalue on $V_\gamma$ is then given by
\begin{equation}
\Cas^{\g,Q}_\gamma:=\langle\gamma,\gamma+2\delta_\g\rangle,
\label{eq:freudenthal}
\end{equation}
where $\delta_\g$ is the half-sum of positive roots of $\g$. We omit the superscript $Q$ if the inner product is clear from context. When working on a normal homogeneous space $M=G/H$ with metric induced by $Q$, we will encounter Casimir operators of both Lie algebras $\g$ and $\h$. Unless otherwise stated, the inner product on $\g$ will be $Q$ and the inner product on $\h$ will be the restriction $Q\big|_{\h\times\h}$.

Suppose $\g$ is a compact Lie algebra, i.e.~$B_\g$ is negative definite, and $\h\subset\g$ is a subalgebra. Fix the standard inner product $-B_\g$ on both $\g$ and $\h$ and consider the adjoint representation $\g$ as a representation of $\h$. An easy calculation then shows that
\begin{equation}
\tr_\g\Cas^{\h,-B_\g}_\g=\dim\h.
\label{eq:castr}
\end{equation}
In particular the Casimir operator of $\g$ on its adjoint representation satisfies the normalization condition
\[\Cas^{\g,-B_\g}_\g=1.\]

On a normal homogeneous space $M=G/H$ the standard curvature endomorphism $q(\Rcr)$ of the reductive connection $\CR$ acts as
\begin{equation}
q(\bar R)=\Cas^\h_V
\label{eq:qrcas}
\end{equation}
%\cite[Lem.~5.2]{AU1}?
on any tensor bundle $VM$. In particular the Ricci endomorphism $\Ricr$ of the reductive connection coincides with $\Cas^\h_\m$. It is well-known that if $g$ is the standard metric, $(M,g)$ is Einstein if and only if $\Cas^\h_\m$ has only one eigenvalue. In this case the Einstein constant $E$ can easily be computed by means of the relation
\begin{equation}
\Cas^\h_\m=2E-\frac{1}{2},
\label{eq:caseinstein}
\end{equation}
cf.~\cite[Prop.~7.89,~7.92]{B87}.

\section{Curvature formulae}
\label{sec:curv}

In order to compute the endomorphism $q(\Riem)$ on a normal homogeneous space, we would like to relate it to the curvature endomorphism $q(\Rcr)$ of the reductive connection, which coincides with the Casimir operator $\Cas^\h$ on the fiber. As mentioned in Section~\ref{sec:prelimnh}, the reductive connection is an instance of a metric connection $\CR$ with parallel skew torsion $\Tcr$. Such a connection can always be recovered from its torsion by means of the formula $\CR=\LC+\frac{1}{2}\Tcr$. Moreover there is a well-known relation (cf.~\cite{CMS})
\begin{equation}
(\Riem-\Rcr)(X,Y)=\frac{1}{4}([\Tcr_X,\Tcr_Y]-2\Tcr_{\Tcr_XY})
\label{eq:curvskew}
\end{equation}
between its curvature tensor $\Rcr$ and the Riemannian curvature $\Riem$, where $\Tcr_X:=\Tcr(X,\cdot)$.

Note that with $\A_XY=[X,Y]_\m$ the torsion of the reductive connection is given by $\Tcr=-\A$. Despite only the case of the reductive connection being necessary for our purposes, we state the following lemma in its full generality.

\begin{lem}
Let $(M,g)$ be a Riemannian manifold with Levi-Civita connection $\LC$ and another metric connection $\CR=\LC+\frac{1}{2}\Tcr$ with parallel skew torsion. On symmetric tensors of any rank,
\[q(\Riem)-q(\Rcr)=-\frac{1}{4}\sum_i(\Tcr_{e_i})_\ast^2.\]
\end{lem}
\begin{proof}
Let $(e_i)$ be an orthonormal basis of $T_xM$ and denote $a_{ijk}:=g(\Tcr(e_i,e_j),e_k)$. Note that $a_{ijk}$ is antisymmetric in the indices $i,j,k$. It follows from the definition of the curvature endomorphism and equation~(\ref{eq:curvskew}) that
\begin{align*}
q(\Riem)-q(\Rcr)&=\frac{1}{4}\sum_{j<k}(e_j\wedge e_k)_\ast([\Tcr_{e_j},\Tcr_{e_k}]-2\Tcr_{\Tcr_{e_j}e_k})_\ast.
\end{align*}
Looking at the individual terms,
\begin{align*}
[\Tcr_{e_j},\Tcr_{e_k}]&=\sum_{\substack{i\\l<m}}(a_{kli}a_{jim}-a_{jli}a_{kim})e_l\wedge e_m,\\
\Tcr_{\Tcr_{e_j}e_k}&=\sum_{\substack{i\\l<m}}a_{jki}a_{ilm}e_l\wedge e_m.
\end{align*}
It follows that
\[\sum_i(\Tcr_{e_i})_\ast^2=\sum_{\substack{i\\j<k\\l<m}}a_{ijk}a_{ilm}(e_j\wedge e_k)_\ast(e_l\wedge e_m)_\ast=\sum_{j<k}(e_j\wedge e_k)_\ast(\Tcr_{\Tcr_{e_j}e_k})_\ast.\]
Let $S(X,Y):=[\Tcr_X,\Tcr_Y]-\Tcr_{\Tcr_XY}$. It remains to show that $q(S)=0$ on symmetric tensors. Indeed, 
\begin{align*}
q(S)e_k&=\sum_{i<j}(e_i\wedge e_j)_\ast([\Tcr_{e_i},\Tcr_{e_j}]e_k-\Tcr_{\Tcr_{e_i}e_j}e_k)\\
&=\sum_{i,j}g([\Tcr_{e_i},\Tcr_{e_j}]e_k-\Tcr_{\Tcr_{e_i}e_j}e_k,e_i)e_j\\
&=\sum_{i,j,l}(a_{ili}a_{jkl}-a_{jli}a_{ikl}-a_{ijl}a_{lki})e_j=0
\end{align*}
using the antisymmetry of $a_{ijk}$, so $q(S)$ vanishes on $T_xM$. Let now $p\in\N$ and denote by $\odot$ the associative symmetric product. For $X_1,\ldots,X_p\in T_xM$,
\begin{align*}
q(S)(X_1\odot\ldots\odot X_p)&=\sum_{i<j}(e_i\wedge e_j)_\ast S(e_i\wedge e_j)_\ast(X_1\odot\ldots\odot X_p)\\
&=\sum_{\substack{i<j\\k}}X_1\odot\ldots\odot(e_i\wedge e_j)_\ast S(e_i\wedge e_j)X_k\odot\ldots\odot X_p\\
&\phantom{=}+\sum_{\substack{i<j\\k\neq l}}X_1\odot\ldots\odot(e_i\wedge e_j)_\ast X_k\odot\ldots\odot S(e_i\wedge e_j)_\ast X_l\odot\ldots\odot X_p.
\end{align*}
Summing over $i,j$ in the first sum reduces it to having a factor of the form $q(S)X$ in each summand, which was just shown to vanish. The second sum, on the other hand, can be grouped to contain factors of the type 
\begin{align*}
&\sum_{i<j}((e_i\wedge e_j)_\ast X\odot S(e_i,e_j)Y+S(e_i,e_j)X\odot(e_i\wedge e_j)_\ast Y)\\
=\,&\sum_{i,j}(g(e_i,X)e_j\odot S(e_i,e_j)Y+g(e_i,Y)S(e_i,e_j)X\odot e_j)\\
=\,&\sum_{j}e_j\odot(S(X,e_j)Y+S(Y,e_j)X),
\end{align*}
which vanishes as well since
\begin{align*}
\langle S(e_i,e_j)e_k+S(e_k,e_j)e_i,e_l\rangle=&\sum_m(a_{jkm}a_{iml}-a_{ikm}a_{jml}-a_{ijm}a_{mkl}\\
&+a_{jim}a_{kml}-a_{kim}a_{jml}-a_{kjm}a_{mil})=0.
\end{align*}
Combining the above, we obtain $q(S)=0$ on $\Sym^pT_xM$. The same calculation works for $\Sym^pT_x^\ast M$ up to sign changes in the action of $\so(T_xM)$, which however cancel out in the end. In total, this proves the assertion.
\end{proof}

\begin{kor}\label{qdiff}
If $(M,g)$ is normal homogeneous with reductive connection $\CR$, then
\[q(\Riem)-q(\Rcr)=\frac{1}{4}\A^\ast\A.\]
\end{kor}

From now on, we stay in the normal homogeneous setting as introduced in Section \ref{sec:prelimnh}, where $\CR$ is the reductive connection and $\A_XY=[X,Y]_\m$. The $H$-equivariant endomorphism $\A^\ast\A$ can itself be written in terms of Casimir operators, yielding an approach to the computation of its spectrum. For $p\in\N$, consider the $p$-fold tensor power $\m^{\otimes p}$ embedded into $\g^{\otimes p}$, and let
\[\pr_{\m^{\otimes p}}:\ \g^{\otimes p}\longrightarrow\m^{\otimes p}\]
be the orthogonal projection onto $\m^{\otimes p}$ with respect to the inner product naturally induced by $Q$ on the tensor power.

\begin{lem}
\label{acas1}
On tensors of rank $p$,
\[\A^\ast\A=\pr_{\m^{\otimes p}}\Cas^\g_{\g^{\otimes p}}\big|_{\m^{\otimes p}}-\Cas^\h_{\m^{\otimes p}}-\Der_{\Cas^\h_\m}.\]
\end{lem}
\begin{proof}
Let $X\in\m$. Since $[\m,\h]\subset\m$, the operator $\ad(X)\in\so(\g)$ can be written as a block matrix
\[\ad(X)=\begin{pmatrix}0&r_X'\\r_X&\A_X\end{pmatrix}\]
with respect to the decomposition $\g=\h\oplus\m$, where
\[r_X=\ad(X)\big|_\h\quad\text{and}\quad r_X'=-(r_X)^\ast=\pr_\h\ad(X)\big|_\m.\]
Consider now the $p$-fold tensor power
\begin{equation}
\g^{\otimes p}=(\h\oplus\m)^{\otimes p}=\bigoplus_{r=0}^p\V_r\label{ptensordecomp}
\end{equation}
where $\V_r\cong\binom{p}{r}\h^{\otimes p-r}\otimes\m^{\otimes r}$. In particular $\V_p=\m^{\otimes p}$. Note that the induced endomorphism $\ad(X)_\ast\in\so(\g^{\otimes p})$ is a derivation, changing only one factor in the tensor product at once. Hence
\[\ad(X)_\ast:\ \V_r\to\V_{r-1}\oplus\V_r\oplus\V_{r+1}\]
(we set $\V_{-1}=\V_{p+1}=0$). In other words, it takes the block form
\[\ad(X)_\ast=\begin{pmatrix}
                  0&\ast&0&\ldots&0\\
                  \ast&\ast&\ddots&\ddots&\vdots\\
                  0&\ddots&\ddots&\ddots&0\\
                  \vdots&\ddots&\ddots&\ast&\ast\\
                  0&\ldots&0&\ast&(\A_X)_\ast
                 \end{pmatrix}
\]
with respect to decomposition (\ref{ptensordecomp}). The nonzero entries of the last row and column are given by
\begin{align*}
a_{p-1,p}&=\pr_{\V_{p-1}}\ad(X)_\ast\big|_{\m^{\otimes p}}=(r_X')_\ast,\\
a_{p,p-1}&=\pr_{\m^{\otimes p}}\ad(X)_\ast\big|_{\V_{p-1}}=r_X\otimes\Id_{\m^{\otimes(p-1)}},\\
a_{p,p}&=\pr_{\m^{\otimes p}}\ad(X)_\ast\big|_{\m^{\otimes p}}=(\A_X)_\ast.
\end{align*}
Combining these, the lowest rightmost entry of $\ad(X)_\ast^2$ is
\[\pr_{\m^{\otimes p}}\ad(X)_\ast^2\big|_{\m^{\otimes p}}=(r_X\otimes\Id_{\m^{\otimes(p-1)}})\circ(r_X')_\ast+(\A_X)_\ast^2.\]
For $X_1,\ldots,X_p\in\m^{\otimes p}$, we have
\begin{align*}
&(r_X\otimes\Id_{\m^{\otimes(p-1)}})\circ(r_X')_\ast(X_1\otimes\ldots\otimes X_p)\\
=\,&(r_X\otimes\Id_{\m^{\otimes(p-1)}})([X,X_1]_\h\otimes X_2\otimes\ldots\otimes X_p+\ldots+X_1\otimes\ldots\otimes X_{p-1}\otimes[X,X_p]_\h)\\
=\,&[X,[X,X_1]_\h]\otimes X_2\otimes\ldots\otimes X_p+\ldots+X_1\otimes\ldots\otimes X_{p-1}\otimes[X,[X,X_p]_\h]\\
=\,&\Der_{[X,[X,\cdot]_\h]}(X_1\otimes\ldots\otimes X_p).
\end{align*}
Together with (\ref{eq:crcurv}) and (\ref{eq:qrcas}) this implies
\begin{align*}
\sum_i(r_{e_i}\otimes\Id_{\m^{\otimes(p-1)}})\circ(r_{e_i}')_\ast&=\sum_i\Der_{[e_i,[e_i,\cdot]_\h]}=-\Der_{\Ricr}=-\Der_{\Cas^\h_\m},
\end{align*}
where $(e_i)$ is an orthonormal basis of $\m$. Note that $(e_i)$ extends any orthormal basis of $\h$ to an orthonormal basis of $\g$. Thus by definition
\[\Cas^\g_{\g^{\otimes p}}=-\sum_i\ad(e_i)_\ast^2-\Cas^\h_{\g^{\otimes p}}.\]
By virtue of $\m^{\otimes p}\subset\g^{\otimes p}$ being an $H$-invariant subspace,
\[\Cas^\h_{\g^{\otimes p}}\big|_{\m^{\otimes p}}=\Cas^\h_{\m^{\otimes p}}.\]
Putting everything together, we obtain
\begin{align*}
\A^\ast\A&=-\sum_i(\A_{e_i})_\ast^2=-\sum_i(\pr_{\m^{\otimes p}}\ad(e_i)_\ast^2\big|_{\m^{\otimes p}}-(r_{e_i}\otimes\Id_{\m^{\otimes(p-1)}})\circ(r_{e_i}')_\ast)\\
&=\pr_{\m^{\otimes p}}\Cas^\g_{\g^{\otimes p}}\big|_{\m^{\otimes p}}-\Cas^\h_{\m^{\otimes p}}-\Der_{\Cas^\h_\m}.
\end{align*}
\end{proof}

\clearpage
\section{The normal homogeneous space $\E_7/\PSO(8)$}
\label{sec:e7pso8}

We begin with a construction of the exceptional Lie algebra $\e_7$ that has the advantage of introducing the chain of subalgebras $\so(8)\subset\su(8)\subset\e_7$ along the way, which will be important later on. If $\Sym^2_0\R^8$ denotes the space of trace-free symmetric $8\times8$-matrices over $\R$, then
\[\su(8)\longrightarrow\so(8)\oplus\Sym^2_0\R^8:\ X\mapsto(\Re X,\Im X)\]
is a vector space isomorphism. According to the classification of symmetric spaces, there exists a symmetric pair $\su(8)\subset\e_7$ whose complex isotropy representation is equal to $\Lambda^4\C^8$. In other words, there exists an $\SU(8)$-invariant real structure on $\Lambda^4\C^8$, i.e. a real $\SU(8)$-module $W$ such that $\Lambda^4\C^8=W^\C$, and $\e_7=\su(8)\oplus W$.

Upon restriction to $\so(8)\subset\su(8)$, the isotropy representation $W\cong\Lambda^4_+\R^8\oplus\Lambda^4_-\R^8$ decomposes into the self-dual and anti-self-dual forms, which in turn are equivalent to the trace-free second symmetric powers $\Sym^2_0\Sigma^\pm$ of the two half-spin representations $\Sigma^\pm$ (both isomorphic, but not equivalent to the defining representation $\R^8$).

Summarizing this argument, we can construct the exceptional Lie algebra $\e_7$ as a Lie algebra with underlying vector space
\[\e_7:=\so(8)\oplus\m:=\so(8)\oplus(\m_0\oplus\m_1\oplus\m_2),\qquad\m_a:=\Sym^2_0\R^8_a,\quad a=0,1,2,\]
where $\R^8_0,\R^8_1,\R^8_2$ denote the three inequivalent representations of $\so(8)$ on $\R^8$. Due to triality in dimension eight, it is actually immaterial which of the three representations $\R^8_0$, $\R^8_1$ and $\R^8_2$ we identify with the defining representation. Indeed there exists an outer automorphism $\Theta\in\Aut(\so(8))$ of order $3$, which cyclically permutes the $\R^8_a$ and extends to an automorphism of $\e_7$ by cyclically permuting the summands $\m_a$, $a=0,1,2$.

Throughout this and the next chapter we will encounter several different representations of $\so(8)$, $\su(8)$ and $\e_7$ and decompose some of their tensor products. As in Section~\ref{sec:prelimcas} we will label irreducible finite-dimensional complex representations $V_\gamma$ of some Lie algebra by their highest weights $\gamma$. It is therefore appropriate to introduce a basis of fundamental weights for each of the three relevant Lie algebras. Here we follow the convention of Bourbaki \cite[Planches~I, IV, VI]{bourbaki}, using the same sets of fundamental weights in the same order. The fundamental weights are denoted as follows:
\begin{align*}
\omega_1,\ldots,\omega_7&\quad\text{ for }\e_7&&\text{ (type $\E_7$)}&\text{with adjoint representation }V_{\omega_1}&=\e_7,\\
\zeta_1,\ldots,\zeta_7&\quad\text{ for }\su(8)&&\text{ (type $\mathrm{A}_7$)}&\text{with standard representation }V_{\zeta_1}&=\C^8,\\
\eta_1,\ldots,\eta_4&\quad\text{ for }\so(8)&&\text{ (type $\mathrm{D}_4$)}&\text{with standard representation }V_{\eta_1}&=\R^8.
\end{align*}
Under this convention we can write the $\so(8)$-modules $\m_a$ as
\begin{equation}
\m_0=\Sym^2_0\R^8=V_{2\eta_1},\quad\m_1=\Sym^2_0\Sigma^+=V_{2\eta_3},\quad\m_2=\Sym^2_0\Sigma^-=V_{2\eta_4}.
\label{eq:isotropyweights}
\end{equation}
It will become important that precomposing a $\so(8)$-representation with the triality automorphism $\Theta$ cyclically permutes the weights $\eta_1,\eta_3,\eta_4$. The tensor product decompositions in Section~\ref{sec:computation} are computed with the help of the software package \LiE, which uses the same enumerative convention.

In passing we remark that similar to the construction of the exceptional Lie algebras $\g_2$ and $\f_4$, the real division algebra of octonions plays a crucial role in the construction of the Lie algebra $\e_7$. It provides both the automorphism $\Theta$ and the remaining parts of the Lie brackets that are not covered by the action of $\so(8)$.\footnote{Choosing an isometry $\O\cong\R^8$ or, equivalently, an orthonormal basis $(e_1,\ldots,e_8)\subset\O$ with respect to $\langle X,Y\rangle:=\Re(\bar XY)$, we may in fact define a bilinear convolution product $\bullet$ on the vector space $\R^{8\times8}$ by setting
\[A\bullet B:=\sum_{i,j=1}^8A_{ij}L_iBL^\top_j,\]
where $L_i\in\R^{8\times8}$ are the matrices representing the endomorphisms $x\mapsto\overline{e_ix}$. In terms of this convolution product, the triality automorphism on $\so(8)$ reads $\Theta(X):=\frac{1}{4}\Id_{8\times8}\bullet X$, while $[X,Y]:=\frac{1}{2}X\bullet Y$ defines the partial Lie bracket $\Sym^2_0\R^8_0\times\Sym^2_0\R^8_1\to\Sym^2_0\R^8_2$.} For our purposes it suffices to note that the Lie bracket satisfies the commutator relations 
\begin{equation}
[\m_a,\m_b]\subset\m_c\quad\text{for distinct }a,b,c=0,1,2,\qquad[\m_a,\m_a]\subset\so(8).
\label{eq:commrel}
\end{equation}
These properties of the Lie bracket, which is constituted by $\so(8)$-equivariant homomorphisms $\m_a\otimes\m_b\to\e_7$, can be deduced directly using the decompositions of $\m_a\otimes\m_b$ into irreducible $\so(8)$-modules combined with Schur's Lemma (note that the $\m_a$ are self-dual as they are modules of an orthogonal group). For example,
\[\m_0\otimes\m_1\cong V_{2\eta_1+2\eta_3}\oplus V_{\eta_1+\eta_3+\eta_4}\oplus\m_2\]
implies that all $\so(8)$-equivariant homomorphisms $\m_0\otimes\m_1\to\e_7$ must map into $\m_2$, since the two other summands in the above decomposition do not occur as $\so(8)$-submodules of $\e_7$.

We endow $\e_7$ with the standard inner product $-B_{\e_7}$, where $B_{\e_7}$ is the Killing form of $\e_7$, and fix the inner products on $\so(8),\su(8)\subset\e_7$ as the respective restrictions of $-B_{\e_7}$. Given an irreducible representation of any of the three Lie algebras, its Casimir eigenvalue is calculated using Freudenthal's formula (\ref{eq:freudenthal}). The calculation may be implemented with \LiE. The scale factors coming from the choice of inner product on the Lie algebra have to be treated with particular caution. However we can always normalize the result using the Casimir eigenvalues of the adjoint representation, since the ratio $c^\g(V_\gamma):=\Cas^\g_\gamma/\Cas^\g_\g$ is independent of the chosen multiple of the Killing form.

The proper Casimir eigenvalues of the adjoint representations are accessible to us by means of identity (\ref{eq:castr}). Writing the trace in terms of eigenvalues, we have
\begin{align*}
\dim\h=\tr_\g\Cas^\h_\g=\sum_i\dim\g_i\cdot\Cas^\h_{\g_i},
%\dim\so(8)&=\tr_{\e_7}\Cas^{\so(8)}_{\e_7}=\dim\so(8)\cdot\Cas^{\so(8)}_{\so(8)}+\sum_{a=0}^2\dim\m_a\cdot\Cas^{\so(8)}_{\m_a}.
\end{align*}
where $\g=\bigoplus_i\g_i$ is a decomposition into irreducible $\h$-modules. Note that in the cases we are interested in, $\h$ is simple, so $\Cas^\h_\h$ can be treated as a constant. This constant can now be expressed as
\[\Cas^\h_\h=\frac{\dim\h}{\sum_i\dim\g_i\cdot c^\h(\g_i)}.\]
The ratios $c^\h(\g_i)$ on the right hand side can now be computed with whatever inner product on $\h$ is convenient. We ultimately arrive at the normalizations
\[\Cas^{\e_7}_{\e_7}=1,\quad\Cas^{\su(8)}_{\su(8)}=\frac{4}{9},\quad\Cas^{\so(8)}_{\so(8)}=\frac{1}{6}.\]
Furthermore we find that the modules $\m_a$ have the same Casimir eigenvalue $\Cas^{\so(8)}_{\m_a}=\frac{2}{9}$, $a=0,1,2$, which is expected as the Casimir operator is invariant under automorphisms of the Lie algebra.

Let $\E_7:=\Aut^0(\e_7)\subset\SO(\e_7)$ be the compact adjoint form of $\e_7$. The unique simply connected compact Lie group $\widetilde \E_7$ with Lie algebra $\e_7$, which is the $2$-fold universal cover of $\E_7$, can be constructed as the preimage $\widetilde \E_7\subset\Spin(\e_7)$ under the spin covering. Inside both $\E_7$ and $\widetilde\E_7$ one finds the projective special orthogonal group
\[\PSO(8):=\SO(8)/_{\displaystyle\{\pm\Id\}}\]
as the unique connected subgroup with Lie algebra $\so(8)$. In consequence there are actually two connected homogeneous spaces
\[M:=\E_7/_{\displaystyle\PSO(8)}=\widetilde \E_7/_{\displaystyle\PSO(8)\times\Z_2},\qquad\widetilde M:=\widetilde \E_7/_{\displaystyle\PSO(8)}\]
representing the pair $\so(8)\subset\e_7$ of Lie algebras, the latter the universal cover of the former. Note that $\SO(8)$ is not contained in either $\E_7$ or $\widetilde \E_7$.
%Even though not every representation of $\e_7$ descends to a representation of $\E_7$, the stability of the simply connected space implies that of the other.

Let $g$ denote the standard metric induced by $-B_{\e_7}$ on both $M$ and $\widetilde M$. The Casimir operator $\Cas^{\so(8)}_\m$ of the isotropy representation is a multiple of the identity, namely $\frac{2}{9}$, so $g$ is an Einstein metric with Einstein constant $E=\frac{13}{36}$ by virtue of (\ref{eq:caseinstein}). Since there is a decomposition of the isotropy representation $\m=\m_0\oplus\m_1\oplus\m_2$ into three pairwise orthogonal $\PSO(8)$-modules satisfying the commutator relations (\ref{eq:commrel}), the normal homogeneous space $(M,g)$ is a so-called \emph{generalized Wallach space} (see \cite{LNF04}).

Finally we note that the normal homogenous space $\E_7/\PSO(8)$ is the total space of a totally geodesic Riemannian submersion with the symmetric space $\E_7/(\SU(8)/\Z_4)$ as base. Notably, the fiber $(\SU(8)/\SO(8))/\Z_2$ is itself locally symmetric. It is easy to check that the conditions of \cite[Thm.~9.73]{B87} for the existence of a second Einstein metric in the canonical variation of metrics are satisfied. This is again an invariant Einstein metric belonging to the $3$-dimensional family of invariant metrics on $\E_7/\PSO(8)$. In fact there are three distinct such submersions with vertical tangent spaces $\m_0$, $\m_1$ and $\m_2$, respectively, yielding three invariant Einstein metrics on $\E_7/\PSO(8)$ besides the normal one. The $G$-instability of those was shown in \cite{L2}, but also follows from results of \cite{WW21}.

 \clearpage
\section{The spectrum of the standard curvature endomorphism}
\label{sec:computation}

In this section we calculate the eigenvalues and eigenspaces of the auxiliary curvature term $\A^\ast\A$ and thus, via Corollary~\ref{qdiff}, the standard curvature endomorphism $q(R)$ on the fiber $\Sym^2\m^\ast$ of the vector bundle $\Sym^2T^\ast M$ over the base point of the homogeneous space $M=\E_7/\PSO(8)$ or its universal cover. The minimal eigenvalue of $q(R)$ will then give a lower bound for the Lichnerowicz Laplacian $\Delta_L$, concluding the proof of Theorem \ref{lest}. All subsequent calculations use the standard Riemannian metric $g$ with Einstein constant $E=\frac{13}{16}$ as defined in Section~\ref{sec:e7pso8}. For any other normal metric $\frac{1}{c}g$ on $M$, the eigenvalues have to be multiplied by $c>0$.

In order to compute the spectrum of the $\PSO(8)$-equivariant endomorphism $\A^\ast\A$, we exploit the inclusions $\so(8)\subset\su(8)\subset\e_7$. In fact there are several distinct intermediate subalgebras of type $\su(8)$, exhibiting a certain symmetry under triality.

\begin{defn}
\label{su8a}
For $a=0,1,2$, let $\su(8)_a:=\so(8)\oplus\m_a$. By (\ref{eq:commrel}), these are Lie subalgebras of $\e_7$ which are isomorphic to one another via the triality automorphism $\Theta$. Denote by $\m_a^\perp$ the orthogonal complement of $\m_a\subset\m$. We define a representation of $\su(8)_a$ on $\m=\m_a\oplus\m_a^\perp$ as follows:
\begin{enumerate}[(i)]
 \item On $\m_a$ the Lie algebra $\su(8)_a$ acts trivially.
 \item On $\m_a^\perp$ the Lie algebra $\su(8)_a$ acts through the Lie bracket of $\e_7$.\footnote{This is well-defined by (\ref{eq:commrel}) since both $\so(8)$ and $\m_a$ preserve $\m_a^\perp=\m_b\oplus\m_c$ ($a,b,c$ distinct) under the Lie bracket of $\e_7$.}
 \item Further, when $\so(8)\subset\su(8)_a$ acts on $\m$ through restriction of the action defined above, we indicate this by the subscript $\so(8)_a$.
 %Via $\Theta$, $\m_a^\perp$ is isomorphic to the $\su(8)$-module $W$ from Section~\ref{sec:e7pso8}. thus inherit this action 
\end{enumerate}
\end{defn}

\begin{lem}
\label{acas2}
On any tensor bundle over $\E_7/\PSO(8)$,
\[\A^\ast\A=(\A^\ast\A)_0+(\A^\ast\A)_1+(\A^\ast\A)_2\quad\text{where}\quad(\A^\ast\A)_a=\Cas^{\su(8)_a}-\Cas^{\so(8)_a}.\]
Moreover the endomorphism $(\A^\ast\A)_0$ determines the other parts by
\[(\A^\ast\A)_{a+1}=\Theta_\ast^{-1}\circ(\A^\ast\A)_a\circ\Theta_\ast,\qquad a\in\Z_3.\]
\end{lem}
\begin{proof}
Recall that $\A^\ast\A$ is defined as a sum over an orthonormal basis of the isotropy representation $\m$, which has the invariant orthogonal decomposition $\m=\m_0\oplus\m_1\oplus\m_2$. In turn we can write $\A^\ast\A$ as a sum
\[\A^\ast\A=(\A^\ast\A)_0+(\A^\ast\A)_1+(\A^\ast\A)_2\]
of $\PSO(8)$-equivariant self-adjoint endomorphisms $(\A^\ast\A)_a$ defined by summing over an orthonormal basis $(e_i^{(a)})$ of $\m_a$, i.e.
\[(\A^\ast\A)_a:=-\sum_i\big(\A_{e_i^{(a)}}\big)_\ast^2.\]
The extended triality automorphism $\Theta\in\Aut(\e_7)$ maps the subspace $\m\subset\e_7$ isometrically to itself and permutes $\m_0$, $\m_1$ and $\m_2$. In consequence $\Theta$ preserves $\A$, that is,
\[\Theta(\A_XY)=\Theta([X,Y]_\m)=[\Theta X,\Theta Y]_\m=\A_{\Theta X}(\Theta Y),\]
and maps any orthonormal basis of $\m_a$ to an orthonormal basis of $\m_{a-1}$. It is then easy to see that
\[(\A^\ast\A)_{a+1}=\Theta^{-1}\circ(\A^\ast\A)_a\circ\Theta,\qquad a\in\Z_3,\]
holds on $\m$. Provided we replace $\Theta$ with its induced action $\Theta_\ast$ on tensors, these relations continue to hold on tensor powers of $\m$, 

Let now $X\in\m_a$ and $Y\in\m_b$ for some $a,b=0,1,2$. By the commutator relations (\ref{eq:commrel}),
\[\A_XY=[X,Y]_\m=\begin{cases}
                            0&a=b,\\
                            [X,Y]&a\neq b.
                           \end{cases}
\]
This means $X$ acts on $\m$ through the $\su(8)_a$-action defined in \ref{su8a}. Completing $(e_i^{(a)})$ to an orthonormal basis of $\su(8)_a$, we immediately obtain
\[\Cas^{\su(8)_a}=(\A^\ast\A)_a+\Cas^{\so(8)_a}.\]
\end{proof}

\setcounter{paragraph}{2}
\paragraph{Isotypical decomposition of $\Sym^2\m$.}
Since $\A^\ast\A$ is a symmetric $\PSO(8)$-equi\-va\-ri\-ant endomorphism of $\Sym^2\m^\ast$, all its ei\-gen\-spa\-ces are necessarily $\PSO(8)$-invariant subspaces. Hence we will begin by decomposing the fiber $\Sym^2\m^\ast$ into isotypical subspaces. We note that $\m^\ast\cong\m$ is self-dual via the invariant inner product, thus also $\Sym^2\m\cong\Sym^2\m^\ast$. The second symmetric power of $\m=\m_0\oplus\m_1\oplus\m_2$ initially decomposes as
\[\Sym^2\m=\Sym^2\m_0\oplus\Sym^2\m_1\oplus\Sym^2\m_2\oplus(\m_0\otimes\m_1)\oplus(\m_0\otimes\m_2)\oplus(\m_1\otimes\m_2).\]
Recall the description (\ref{eq:isotropyweights}) of the $\PSO(8)$-modules $\m_a$ in terms of highest weights. With help of \LiE, the above decomposition can be refined as follows:
\begin{equation}
\begin{aligned}
\Sym^2\m_0&=\Sym^2V_{2\eta_1}=\R\oplus V_{4\eta_1}\oplus V_{2\eta_1}\oplus V_{2\eta_2},\\
\Sym^2\m_1&=\Sym^2V_{2\eta_3}=\R\oplus V_{4\eta_3}\oplus V_{2\eta_3}\oplus V_{2\eta_2},\\
\Sym^2\m_2&=\Sym^2V_{2\eta_4}=\R\oplus V_{4\eta_4}\oplus V_{2\eta_4}\oplus V_{2\eta_2},\\
\m_0\otimes\m_1&=V_{2\eta_1}\otimes V_{2\eta_3}=V_{2\eta_1+2\eta_3}\oplus V_{\eta_1+\eta_3+\eta_4}\oplus V_{2\eta_4},\\
\m_0\otimes\m_2&=V_{2\eta_1}\otimes V_{2\eta_4}=V_{2\eta_1+2\eta_4}\oplus V_{\eta_1+\eta_3+\eta_4}\oplus V_{2\eta_3},\\
\m_1\otimes\m_2&=V_{2\eta_3}\otimes V_{2\eta_4}=V_{2\eta_3+2\eta_4}\oplus V_{\eta_1+\eta_3+\eta_4}\oplus V_{2\eta_1}.
\end{aligned}
\label{eq:sym2decomp}
\end{equation}
Note that the last three lines imply the relations $[\m_a,\m_b]\subset\m_c$ in (\ref{eq:commrel}) for the $\e_7$ Lie bracket, as hinted at in Section~\ref{sec:e7pso8}. Note also the symmetry under triality, i.e.~under permutation of the weights $\eta_1$, $\eta_3$ and $\eta_4$.

These highest weight modules can be further interpreted as
\begin{align*}
V_{4\eta_1}&=\m_0\cartan\m_0=\Sym^4_0\R^8_0,&V_{2\eta_2}&=\Lambda^2\R^8_0\cartan\Lambda^2\R^8_0,\\
V_{\eta_1+\eta_3+\eta_4}&=\R^8_0\cartan\Lambda^3\R^8_0,&V_{2\eta_1+2\eta_3}&=\m_0\cartan\m_1
\end{align*}
and similarly for permutations of $\eta_1,\eta_3,\eta_4$ (resp. $\R^8_0,\R^8_1,\R^8_2$). Here, $\cartan$ denotes the Cartan product of irreducible representations,
\[V_\gamma\cartan V_{\gamma'}:=V_{\gamma+\gamma'}\subset V_{\gamma}\otimes V_{\gamma'}.\]
Moreover $V_{2\eta_2}$ can be identified with the space of algebraic Weyl tensors over any of the $8$-dimensional representations of $\so(8)$.

\paragraph{Actions of $\e_7$ and $\su(8)$.}
In order to compute the spectrum of $\A^\ast\A$ on $\Sym^2\m$ by means of Lemmas~\ref{acas1} and \ref{acas2}, one needs to evaluate Casimir operators of $\e_7$, $\su(8)_a$ and $\so(8)_a$. It is therefore essential to identify how these Lie algebras act on each isotypical summand of $\Sym^2\m$, or, to be more precise, how each summand embeds into a module of each $\e_7$, $\su(8)_a$ and $\so(8)_a$. We thus turn to decompositions of suitable modules that are invariant under $\e_7$ or $\su(8)$, respectively. All subsequent decompositions and branchings to subalgebras are computed with help of \LiE.

First, consider the embedding $\Sym^2\m\subset\Sym^2\e_7$. The right hand side decomposes into irreducible $\e_7$-modules as
\[\Sym^2\e_7=\R\oplus V_{\omega_6}\oplus V_{2\omega_1}.\]
Branching to $\so(8)$ gives
\begin{equation}
\begin{aligned}
V_{\omega_6}\cong&\,V_{2\eta_1}\oplus V_{2\eta_3}\oplus V_{2\eta_4}\oplus 3V_{\eta_1+\eta_3+\eta_4}\oplus V_{2\eta_2}\oplus 3V_{\eta_2},\\
V_{2\omega_1}\cong&\,3\R\oplus 3V_{2\eta_1}\oplus 3V_{2\eta_3}\oplus 3V_{2\eta_4}\oplus V_{4\eta_1}\oplus V_{4\eta_3}\oplus V_{4\eta_4}\oplus 3V_{\eta_1+\eta_3+\eta_4}\oplus3V_{2\eta_2}\\
&\oplus V_{2\eta_1+2\eta_3}\oplus V_{2\eta_1+2\eta_4}\oplus V_{2\eta_3+2\eta_4}\oplus V_{\eta_2+2\eta_1}\oplus V_{\eta_2+2\eta_3}\oplus V_{\eta_2+2\eta_4}.
\end{aligned}
\label{eq:sym2e7decomp}
\end{equation}
By comparison with (\ref{eq:sym2decomp}), we find that the summands $\m_a\cartan\m_b$ necessarily embed into $V_{2\omega_1}$ for $a,b=0,1,2$. Moreover, by considering the tracefree part $\Sym^2_0\e_7\cong V_{\omega_6}\oplus V_{2\omega_1}$, we see that the $2$-dimensional trivial submodule $(\Sym^2_0\m)^{\so(8)}\cong 2\R$ of $\Sym^2_0\m$ also lies inside $V_{2\omega_1}$. Thus on these summands the $\Cas^{\e_7}_{\e_7\otimes\e_7}$-term from Lemma~\ref{acas1} is simply multiplication by the constant $\Cas^{\e_7}_{2\omega_1}$.

Second, recall that as an $\su(8)_a$-representation
\[\m=\m_a\oplus\m_a^\perp,\qquad a=0,1,2,\]
where $\m_a$ is trivial and $\m_a^\perp\cong W$ with $W^\C\cong\Lambda^4\C^8=V_{\zeta_4}$. Thus
\begin{align*}
\Sym^2\m&=\underbrace{\Sym^2\m_a}_{\text{trivial}}\oplus(\underbrace{\m_a\otimes\m_a^\perp}_{\cong35V_{\zeta_4}})\oplus\Sym^2\m_a^\perp,\\
\Sym^2\m_a^\perp&=\R\oplus V_{2\zeta_4}\oplus V_{\zeta_2+\zeta_6}.
\end{align*}
Since $\so(8)_0$ is embedded into $\su(8)_0$ in the standard way, the branchings of the $\su(8)_0$-representations $V_{2\zeta_4},V_{\zeta_2+\zeta_6}$ to $\so(8)_0$ can easily be computed:
\begin{equation}
\begin{aligned}
V_{2\zeta_4}&\cong V_{4\eta_3}\oplus V_{4\eta_4}\oplus V_{2\eta_3+2\eta_4}\oplus V_{2\eta_2}\oplus V_{2\eta_1}\oplus\R,\\
V_{\zeta_2+\zeta_6}&\cong V_{2\eta_3}\oplus V_{2\eta_4}\oplus V_{2\eta_2}\oplus V_{\eta_1+\eta_3+\eta_4}.
\end{aligned}
\label{eq:sym2su8decomp}
\end{equation}
The branchings of $\su(8)_{1,2}$-representations to $\so(8)_{1,2}$ work similarly, but with $\eta_1,\eta_3,\eta_4$ permuted by triality. Comparing with the isotypical decomposition of $\Sym^2\m$, we can again identify the actions of $\su(8)_a$ as well as $\so(8)_a$ on some summands of (\ref{eq:sym2decomp}). The results are collected in Table \ref{decomp}. Whenever a summand of $\Sym^2\m$ embeds into a unique isotypical module $V_\gamma$ of $\e_7$ or $\su(8)_a$, the corresponding Casimir operator acts as multiplication by the constant $\Cas_\gamma$. In each of those cases this constant is computed using \LiE\ and listed in Table \ref{allcas}.

\paragraph{Eigenvalues of $\A^\ast\A$ on remaining components.}
On any isotypical summand of $\Sym^2\m$ where the preceding has shown that the Casimir operators of either $\e_7$ or $\su(8)_a$ are multiples of the identity, we find the eigenvalue of $\A^\ast\A$ by one of the formulas from Lemmas~\ref{acas1} and \ref{acas2} (see Table \ref{alleig}). This works for most summands of $\Sym^2\m$, except for
\begin{enumerate}[(i)]
 \item the three copies of the representation $V_{2\eta_2}$ of Weyl tensors on $\R^8$,
 \item the trace part in $\Sym^2\m$, i.e.~the trivial summand spanned by $B_{\e_7}\big|_{\m}$.
\end{enumerate}
Issue (ii) is swiftly remedied by noting that
\[\A_Xg(Y,Z)=-g(\A_XY,Z)-g(Y,\A_XZ)=0\]
since the $(2,1)$-tensor $\A$ is totally skew-symmetric, and thus $\A^\ast\A g=0$. However (i) requires a more careful analysis.

Denote by $\Weyl_a$ the copy of $V_{2\eta_2}$ occurring inside $\Sym^2\m_a$, cf.~(\ref{eq:sym2decomp}), and
\[\Weyl:=\Weyl_0\oplus\Weyl_1\oplus\Weyl_2\subset\Sym^2\m,\qquad\Weyl\cong 3V_{2\eta_2}.\]
Consider the operator $(\A^\ast\A)_0$ on $\Weyl$. By the construction of $(\A^\ast\A)_0$ and commutator relations (\ref{eq:commrel}), $(\A^\ast\A)_0$ must annihilate $\Weyl_0\subset\Sym^2\m_0$ and preserve $\Weyl_1\oplus\Weyl_2$. Combined with the symmetries under triality, it follows that $(\A^\ast\A)_0$ takes the block form
\begin{equation}
(\A^\ast\A)_0\big|_{\Weyl}=\begin{pmatrix}
                     0&0&0\\
                     0&s&t\\
                     0&t&s
                    \end{pmatrix},\qquad s,t\in\R,
\label{eq:a0block}
\end{equation}
with respect to the above decomposition of $\Weyl$. This matrix has eigenvalues $0$ and $s\pm t$. Since $(\A^\ast\A)_0$ determines $(\A^\ast\A)_1$ and $(\A^\ast\A)_2$ by triality, these have a similar block form. Summing up, we find that the matrix of $\A^\ast\A$ reads
\[\A^\ast\A\big|_{\Weyl}=(\A^\ast\A)_0\big|_{\Weyl}+(\A^\ast\A)_1\big|_{\Weyl}+(\A^\ast\A)_2\big|_{\Weyl}=\begin{pmatrix}
                           2s&t&t\\
                           t&2s&t\\
                           t&t&2s
                          \end{pmatrix}.
\]
We look for clues to determine $s$ and $t$. First, recall that $\su(8)_0$ acts on $\Sym^2\m_0^\perp$ through the Lie bracket, thus $\Weyl_1\oplus\Weyl_2\cong 2V_{2\eta_2}$ as an $\so(8)_0$-submodule of $\Sym^2\m_0^\perp$. But the module $V_{2\eta_2}$ occurs with multiplicity $1$ in each of $V_{2\zeta_4},V_{\zeta_2+\zeta_6}\subset\Sym^2\m_0^\perp$, cf. (\ref{eq:sym2su8decomp}). It follows from Lemma~\ref{acas2} that $(\A^\ast\A)_0\big|_{\Weyl_1\oplus\Weyl_2}$ has eigenvalues
\begin{align*}
(\A^\ast\A)_0\big|_{(\Weyl_1\oplus\Weyl_2)\cap V_{2\zeta_4}}&=\Cas^{\su(8)}_{2\zeta_4}-\Cas^{\so(8)}_{2\eta_2}=\frac{4}{9}\cdot\frac{5}{2}-\frac{1}{6}\cdot\frac{7}{3}=\frac{13}{18},\\
(\A^\ast\A)_0\big|_{(\Weyl_1\oplus\Weyl_2)\cap V_{\zeta_2+\zeta_6}}&=\Cas^{\su(8)}_{\zeta_2+\zeta_6}-\Cas^{\so(8)}_{2\eta_2}=\frac{4}{9}\cdot\frac{7}{4}-\frac{1}{6}\cdot\frac{7}{3}=\frac{7}{18}.
\end{align*}
In light of (\ref{eq:a0block}), this implies that $s=\frac{10}{18}$ and $t=\pm\frac{3}{18}$. In turn $\A^\ast\A$ is of block form
\[\A^\ast\A\big|_{\Weyl}=\frac{1}{18}\begin{pmatrix}
                           20&\pm3&\pm3\\
                           \pm3&20&\pm3\\
                           \pm3&\pm3&20
                          \end{pmatrix},
\]
which has eigenvalues $\frac{13}{9},\frac{17}{18},\frac{17}{18}$ or $\frac{23}{18},\frac{23}{18},\frac{7}{9}$.

Second, looking at the decompositions (\ref{eq:sym2e7decomp}), we find that $V_{2\eta_2}$ has multiplicity $4$ in $\Sym^2\e_7$. Denote the $V_{2\eta_2}$-isotypical component of $\Sym^2\e_7$, viewed as an $\so(8)$-module, by $\Weyl'\cong 4V_{2\eta_2}$. Then $\Weyl'\cap V_{2\omega_1}\cong 3V_{2\eta_2}$. Since
\[\Weyl\cap V_{2\omega_1}=\Weyl\cap(\Weyl'\cap V_{2\omega_1})\subset\Weyl'\]
and the intersection of any two $3$-dimensional subspaces in $\R^4$ is at least $2$-dimensional, it follows with Schur's Lemma that $\Weyl\cap V_{2\omega_1}\cong cV_{2\eta_2}$ with $c\geq2$. On this subspace, $\Cas^{\e_7}$ is just multiplication by the constant $\Cas^{\e_7}_{2\omega_1}=\frac{19}{9}$. Thus the eigenvalue of $\A^\ast\A$ is readily computed as
\[\A^\ast\A\big|_{\Weyl\cap V_{2\omega_1}}=\Cas^{\e_7}_{2\omega_1}-\Cas^{\so(8)}_{2\eta_2}-2\Cas^{\so(8)}_{\m}=\frac{19}{9}-\frac{1}{6}\cdot\frac{7}{3}-2\cdot\frac{2}{9}=\frac{23}{18}.\]
Combined with the considerations above, we conclude that $t=-\frac{3}{18}$ and the spectrum of $\A^\ast\A\big|_{\Weyl}$ is given by
\[\frac{7}{9}\text{ on }\diag(V_{2\eta_2})\subset\Weyl,\quad\frac{23}{18}\text{ on }\diag(V_{2\eta_2})^\perp\cong 2V_{2\eta_2}\subset\Weyl.\]

\paragraph{Proof of Theorem~\ref{lest}.}
Now that the spectrum of $\A^\ast\A$ is assembled, we turn to the operator $q(\Riem)$ on $\Sym^2\m$. Recall from (\ref{eq:qrcas}) that $q(\Rcr)=\Cas^{\so(8)}_{\Sym^2\m}$, which is a constant on each isotypical component of $\Sym^2\m$. By virtue of Corollary~\ref{qdiff}, we now obtain $q(\Riem)$ from
\[q(\Riem)=\frac{1}{4}\A^\ast\A+\Cas^{\so(8)}_{\Sym^2\m}.\]
The respective eigenvalues are listed in Table \ref{alleig}. Notice that $q(R)\geq\frac{5}{12}$ on $\Sym^2_0\m$ (excluding the trace part spanned by $B_{\e_7}\big|_\m$), and recall that $E=\frac{13}{36}$. Together with inequality (\ref{eq:weitzenboeck}) this implies that
\[\Delta_L\geq2q(\Riem)\geq\frac{5}{6}=\frac{30}{13}E>2E\]
holds true on $\TT(M)$. Thus the strict stability of the standard metric on $\E_7/\PSO(8)$ is shown.\qed

\setcounter{satz}{6}
\begin{bem}%are fiber and base the right quotients?
This bound on $\Delta_L$ is sharp and realized by $\E_7$-invariant tensors arising from the canonical variation in the three Riemannian submersions
\[\left(\SU(8)/_{\displaystyle\SO(8)}\right)\!\!/_{\displaystyle\Z_2}\longrightarrow\E_7/_{\displaystyle\PSO(8)}\longrightarrow\E_7/_{\displaystyle\SU(8)/\Z_4}\]
with totally geodesic fibres, or, equivalently, from scaling the standard metric on one of the summands in the decomposition $\m=\m_0\oplus\m_1\oplus\m_2$. Indeed, by results of \cite{HMS16}, these are Killing tensors and thus satisfy $\Delta_Lh=2q(\Riem)h$. The Lichnerowicz eigenvalue of these invariant tensors had originally been found by J. Lauret and C. Will \cite[Table 2]{L2}, who also showed that these tensors constitute (up to tracelessness) destabilizing directions for any of the three non-normal Einstein metrics on $\E_7/\PSO(8)$.
\end{bem}

%\afterpage{
\begin{landscape}
\begin{table}[h]
\centering
\renewcommand{\arraystretch}{1.2}
\begin{tabular}{c|c|c|c|c|c|c|c}
&$\su(8)_0$&$\so(8)_0$&$\su(8)_1$&$\so(8)_1$&$\su(8)_2$&$\so(8)_2$&$\e_7$\\\hline
$\m_0\cartan\m_0$&trivial&trivial&$V_{2\zeta_4}$&$V_{4\eta_1}$&$V_{2\zeta_4}$&$V_{4\eta_1}$&$V_{2\omega_1}$\\
$\m_1\cartan\m_1$&$V_{2\zeta_4}$&$V_{4\eta_3}$&trivial&trivial&$V_{2\zeta_4}$&$V_{4\eta_3}$&$V_{2\omega_1}$\\
$\m_2\cartan\m_2$&$V_{2\zeta_4}$&$V_{4\eta_4}$&$V_{2\zeta_4}$&$V_{4\eta_4}$&trivial&trivial&$V_{2\omega_1}$\\
$V_{2\eta_2}\subset\Sym^2\m_0$&trivial&trivial&$V_{2\zeta_4}\oplus V_{\zeta_2+\zeta_6}$&$V_{2\eta_2}$&$V_{2\zeta_4}\oplus V_{\zeta_2+\zeta_6}$&$V_{2\eta_2}$&$V_{2\omega_1}\oplus V_{\omega_6}$\\
$V_{2\eta_2}\subset\Sym^2\m_1$&$V_{2\zeta_4}\oplus V_{\zeta_2+\zeta_6}$&$V_{2\eta_2}$&trivial&trivial&$V_{2\zeta_4}\oplus V_{\zeta_2+\zeta_6}$&$V_{2\eta_2}$&$V_{2\omega_1}\oplus V_{\omega_6}$\\
$V_{2\eta_2}\subset\Sym^2\m_2$&$V_{2\zeta_4}\oplus V_{\zeta_2+\zeta_6}$&$V_{2\eta_2}$&$V_{2\zeta_4}\oplus V_{\zeta_2+\zeta_6}$&$V_{2\eta_2}$&trivial&trivial&$V_{2\omega_1}\oplus V_{\omega_6}$\\
$\m_0\cartan\m_1$&$V_{\zeta_4}$&$V_{2\eta_3}$&$V_{\zeta_4}$&$V_{2\eta_1}$&$V_{2\zeta_4}$&$V_{2\eta_1+2\eta_3}$&$V_{2\omega_1}$\\
$\m_0\cartan\m_2$&$V_{\zeta_4}$&$V_{2\eta_4}$&$V_{2\zeta_4}$&$V_{2\eta_1+2\eta_4}$&$V_{\zeta_4}$&$V_{2\eta_1}$&$V_{2\omega_1}$\\
$\m_1\cartan\m_2$&$V_{2\zeta_4}$&$V_{2\eta_3+2\eta_4}$&$V_{\zeta_4}$&$V_{2\eta_4}$&$V_{\zeta_4}$&$V_{2\eta_3}$&$V_{2\omega_1}$\\
$\m_0\subset\Sym^2\m_0$&trivial&trivial&$V_{\zeta_2+\zeta_6}$&$V_{2\eta_1}$&$V_{\zeta_2+\zeta_6}$&$V_{2\eta_1}$&$V_{2\omega_1}\oplus V_{\omega_6}$\\
$\m_1\subset\Sym^2\m_1$&$V_{\zeta_2+\zeta_6}$&$V_{2\eta_3}$&trivial&trivial&$V_{\zeta_2+\zeta_6}$&$V_{2\eta_3}$&$V_{2\omega_1}\oplus V_{\omega_6}$\\
$\m_2\subset\Sym^2\m_2$&$V_{\zeta_2+\zeta_6}$&$V_{2\eta_4}$&$V_{\zeta_2+\zeta_6}$&$V_{2\eta_4}$&trivial&trivial&$V_{2\omega_1}\oplus V_{\omega_6}$\\
$\m_0\subset\m_1\otimes\m_2$&$V_{2\zeta_4}$&$V_{2\eta_1}$&$V_{\zeta_4}$&$V_{2\eta_4}$&$V_{\zeta_4}$&$V_{2\eta_3}$&$V_{2\omega_1}\oplus V_{\omega_6}$\\
$\m_1\subset\m_0\otimes\m_2$&$V_{\zeta_4}$&$V_{2\eta_4}$&$V_{2\zeta_4}$&$V_{2\eta_3}$&$V_{\zeta_4}$&$V_{2\eta_1}$&$V_{2\omega_1}\oplus V_{\omega_6}$\\
$\m_2\subset\m_0\otimes\m_1$&$V_{\zeta_4}$&$V_{2\eta_3}$&$V_{\zeta_4}$&$V_{2\eta_1}$&$V_{2\zeta_4}$&$V_{2\eta_4}$&$V_{2\omega_1}\oplus V_{\omega_6}$\\
$V_{\eta_1+\eta_3+\eta_4}\subset\Sym^2\m_0$&$V_{\zeta_2+\zeta_6}$&$V_{\eta_1+\eta_3+\eta_4}$&$V_{\zeta_4}$&$V_{2\eta_4}$&$V_{\zeta_4}$&$V_{2\eta_3}$&$V_{2\omega_1}\oplus V_{\omega_6}$\\
$V_{\eta_1+\eta_3+\eta_4}\subset\Sym^2\m_1$&$V_{\zeta_4}$&$V_{2\eta_4}$&$V_{\zeta_2+\zeta_6}$&$V_{\eta_1+\eta_3+\eta_4}$&$V_{\zeta_4}$&$V_{2\eta_1}$&$V_{2\omega_1}\oplus V_{\omega_6}$\\
$V_{\eta_1+\eta_3+\eta_4}\subset\Sym^2\m_2$&$V_{\zeta_4}$&$V_{2\eta_3}$&$V_{\zeta_4}$&$V_{2\eta_1}$&$V_{\zeta_2+\zeta_6}$&$V_{\eta_1+\eta_3+\eta_4}$&$V_{2\omega_1}\oplus V_{\omega_6}$\\
$(\Sym^2_0\m)^{\so(8)}$&$\R\oplus V_{2\zeta_4}$&trivial&$\R\oplus V_{2\zeta_4}$&trivial&$\R\oplus V_{2\zeta_4}$&trivial&$V_{2\omega_1}$\\
$\R B_{\e_7}\big|_\m$&$\R\oplus V_{2\zeta_4}$&trivial&$\R\oplus V_{2\zeta_4}$&trivial&$\R\oplus V_{2\zeta_4}$&trivial&$\R\oplus V_{2\omega_1}$
\end{tabular}
\caption{All $\so(8)$-irreducible summands of $\Sym^2\m$ and the highest weight modules they embed into.}
\label{decomp}
\end{table}
\end{landscape}
%}

\begin{table}[ht]
\centering
\renewcommand{\arraystretch}{1.2}
\begin{tabular}{c|c|c|c|c|c|c}
&$\Cas^{\su(8)_a}$&$\Cas^{\so(8)_a}$&$\Cas^{\su(8)_b}$&$\Cas^{\so(8)_b}$&$\Cas^{\e_7}$&$\Cas^{\so(8)}$\\\hline
$\m_a\cartan\m_a$&$0$&$0$&$\frac{10}{9}$&$\frac{5}{9}$&$\frac{19}{9}$&$\frac{5}{9}$\\
$V_{2\eta_2}\subset\Sym^2\m_a$&$0$&$0$&--&$\frac{7}{18}$&--&$\frac{7}{18}$\\
$\m_a\cartan\m_c$&$\frac{1}{2}$&$\frac{2}{9}$&$\frac{10}{9}$&$\frac{1}{2}$&$\frac{19}{9}$&$\frac{1}{2}$\\
$\m_a\subset\Sym^2\m_a$&$0$&$0$&$\frac{7}{9}$&$\frac{2}{9}$&--&$\frac{2}{9}$\\
$\m_a\subset\m_b\otimes\m_c$&$\frac{10}{9}$&$\frac{2}{9}$&$\frac{1}{2}$&$\frac{2}{9}$&--&$\frac{2}{9}$\\
$V_{\eta_1+\eta_3+\eta_4}\subset\Sym^2\m_a$&$\frac{7}{9}$&$\frac{1}{3}$&$\frac{1}{2}$&$\frac{2}{9}$&--&$\frac{1}{3}$\\
$(\Sym^2_0\m)^{\so(8)}$&--&$0$&--&$0$&$\frac{19}{9}$&$0$\\
$\R B_{\e_7}\big|_\m$&--&$0$&--&$0$&--&$0$\\
\end{tabular}
\caption{Casimir eigenvalues on the summands in Table~\ref{decomp}. Here $a,b,c$ are distinct. A dash indicates that the summand might not be contained in a single eigenspace of the Casimir operator.}
\label{allcas}
\end{table}

\begin{table}[ht]
\centering
\renewcommand{\arraystretch}{1.2}
\begin{tabular}{c|c|c|c|c|c}
&$(\A^\ast\A)_a$&$(\A^\ast\A)_b$&$\A^\ast\A$&$q(\Rcr)$&$q(\Riem)$\\\hline
$\m_a\cartan\m_a$&$0$&$\frac{5}{9}$&$\frac{10}{9}$&$\frac{5}{9}$&$\frac{5}{6}$\\
$\diag(V_{2\eta_2})\subset\Weyl$&--&--&$\frac{7}{9}$&$\frac{7}{18}$&$\frac{7}{12}$\\
$\diag(V_{2\eta_2})^\perp\subset\Weyl$&--&--&$\frac{23}{18}$&$\frac{7}{18}$&$\frac{17}{24}$\\
$\m_a\cartan\m_c$&$\frac{5}{18}$&$\frac{11}{18}$&$\frac{7}{6}$&$\frac{1}{2}$&$\frac{19}{24}$\\
$\m_a\subset\Sym^2\m_a$&$0$&$\frac{5}{9}$&$\frac{10}{9}$&$\frac{2}{9}$&$\frac{1}{2}$\\
$\m_a\subset\m_b\otimes\m_c$&$\frac{8}{9}$&$\frac{5}{18}$&$\frac{13}{9}$&$\frac{2}{9}$&$\frac{7}{12}$\\
$V_{\eta_1+\eta_3+\eta_4}\subset\Sym^2\m_a$&$\frac{4}{9}$&$\frac{5}{18}$&$1$&$\frac{1}{3}$&$\frac{7}{12}$\\
$(\Sym^2_0\m)^{\so(8)}$&--&--&$\frac{5}{3}$&$0$&$\frac{5}{12}$\\
$\R B_{\e_7}\big|_\m$&--&--&$0$&$0$&$0$
\end{tabular}
\caption{The eigenvalues of $\A^\ast\A$, $q(\Rcr)$ and $q(\Riem)$ on the summands of $\Sym^2\m$. Here $a,b,c$ are distinct.}
\label{alleig}
\end{table}

\end{document}